\documentclass[11pt,reqno]{amsart}
\usepackage{palatino}
\usepackage{amsmath,amssymb,amsxtra,color,calligra,mathrsfs,comment,url}

\usepackage{xcolor}
\colorlet{mdtRed}{red!50!black}
\colorlet{dblue}{blue!50!black}

\usepackage[colorlinks, pagebackref=true]{hyperref}
\hypersetup{linkcolor=dblue,citecolor=dblue,filecolor=dullmagenta,urlcolor=mdtRed}
\renewcommand*{\backref}[1]{}
\renewcommand*{\backrefalt}[4]{[{%
	\ifcase #1 Not cited.%
	\or $\uparrow$~#2.%
	\else $\uparrow$~#2.%
	\fi%
}]}
\usepackage[all]{xy}
\usepackage{tikz,tikz-cd,tkz-graph,enumerate}
\usetikzlibrary{matrix,arrows,decorations.pathmorphing}

\usepackage[symbol]{footmisc}

\DeclareMathOperator{\sHom}{\mathcal{H}\!\textit{om}}
\DeclareMathOperator{\sEnd}{\mathcal{E}\!\textit{nd}}

\DeclareMathOperator{\Id}{\textnormal{Id}}
\DeclareMathOperator{\At}{\textnormal{At}}
\DeclareMathOperator{\ad}{\textnormal{ad}}

\DeclareMathOperator{\GL}{\textnormal{GL}}
\DeclareMathOperator{\Der}{\mathcal{D}\!{\it er}}

\newcommand{\mf}[1]{\mathfrak{#1}}
\newcommand{\mc}[1]{\mathcal{#1}}
\newcommand{\scr}[1]{\mathscr{#1}}
\newcommand{\bb}[1]{\mathbb{#1}}

\newtheorem{theorem}{Theorem}[subsection]

\newtheorem{proposition}[theorem]{Proposition}

\theoremstyle{definition}
\newtheorem{definition}[theorem]{Definition}
\newtheorem{remark}[theorem]{Remark}

\numberwithin{equation}{subsection}

\begin{document}

\baselineskip=15.5pt

\title[Logarithmic connection on bundles]{Criterion for existence of a logarithmic connection on a principal bundle 
over a smooth complex projective variety} 

\author[S. Gurjar]{Sudarshan Gurjar}
\address{
	\begin{tabular}{l}
		{Sudarshan Gurjar} \\
		\hspace{.1in} Department of Mathematics,\\
		\hspace{.1in} Indian Institute of Technology Bombay, \\
		\hspace{.1in} Powai, Mumbai 400076, Maharashtra, India. \\
		\hspace{.1in} Email: {\rm \texttt{sgurjar@math.iitb.ac.in}} 
	\end{tabular}
}

\author[A. Paul]{Arjun Paul} 
\address{
	\begin{tabular}{l}
		{Arjun Paul} \\
		\hspace{.1in} Department of Mathematics,\\
		\hspace{.1in} Indian Institute of Technology Bombay, \\
		\hspace{.1in} Powai, Mumbai 400076, Maharashtra, India. \\
		\hspace{.1in} Email: {\rm \texttt{arjunp@math.iitb.ac.in}} 
	\end{tabular}
}



\date{\today}

\thanks{Corresponding author: Arjun Paul}

\begin{abstract}
	Let $X$ be a connected smooth complex projective variety of dimension $n \geq 1$. 
	Let $D$ be a simple normal crossing divisor on $X$. 
	Let $G$ be a connected complex Lie group, and $E_G$ a holomorphic principal $G$-bundle on $X$. 
	In this article, we give criterion for existence of a logarithmic connection on $E_G$ 
	singular along $D$. 
\end{abstract}

\maketitle


\section{Introduction}
A theorem of Weil \cite{Weil-1938} says that a holomorphic vector bundle $E$ 
on a smooth complex projective curve $X$ admits a holomorphic connection if and only if 
each indecomposable holomorphic direct summand of $E$ has degree $0$; see \cite{Atiyah-1957}. 
For connected reductive linear algebraic group $G$ over $\bb C$, this result of Weil and Atiyah 
is generalized to the case of holomorphic principal $G$-bundles on a smooth complex projective 
curve in \cite{Azad-Biswas-2002}. It follows from \cite[Theorem 4, p.~192]{Atiyah-1957} that, 
not every holomorphic bundle on a compact K\"ahler manifold can admit a holomorphic connection. 
Therefore, one can ask for criterion for a holomorphic bundle on $X$ to admit a meromorphic connection. 
Simplest case of meromorphic connection is logarithmic connection. 
So it natural to ask when a given holomorphic bundle on $X$ admits a logarithmic connection singular along a 
given divisor with prescribed residues. When $X$ is a smooth complex projective curve, 
in \cite{Biswas-Dan-Paul-2018}, a necessary and sufficient criterion for a vector bundle on $X$ to admit 
a logarithmic connection singular along a given reduced effective divisor $D$ on $X$ with prescribed rigid 
residues along $D$ is given. This result is further generalized to the case of holomorphic principal 
$G$-bundles over smooth complex projective curve in \cite{Biswas-Dan-Paul-Saha-2017} when $G$ is a 
connected reductive linear algebraic group over $\bb C$. 
When $X$ is a smooth complex projective variety of dimension of more than one, 
no such criterion for existence of logarithmic connection on a holomorphic bundle on $X$ with 
prescribed residues along a given reduced effective divisor is known to the best of our knowledge. 
In this article, we attempt to study this problem. 

\subsection{Outline of the paper} 
Unless otherwise specified, $X$ is a connected smooth complex projective variety of dimension at least one, 
and $D$ a reduced effective divisor on $X$. We denote by $G$ a connected affine algebraic group over $\bb C$. 
In Section \S\ref{sec:preli}, we recall definitions of simple normal crossing divisor, 
logarithmic connection on holomorphic vector bundles on $X$, and their residues along a simple normal 
crossing divisor on $X$. 
In Section \S\ref{sec:log-con-princ-bun}, we extend this notion of logarithmic connection for principal 
$G$-bundles $E_G$ on $X$, and discuss the notion of residue of a logarithmic connection on $E_G$ singular 
along a simple normal crossing divisor on $X$. 

In Section \S\ref{sec:extn-str-gp}, we study logarithmic connection on principal bundles under the 
extensions of structure group. Let $H$ be a connected closed algebraic subgroup of $G$ over $\bb C$. 
Let $E_H$ be a holomorphic principal $H$-bundle on $X$. Let $E_H(G)$ be the holomorphic principal 
$G$-bundle on $X$ obtained by extending the structure group of $E_H$ by the inclusion map $H \subset G$. 
Then we have the following (see Proposition \ref{prop:red-str-gp}). 
\begin{proposition}
	If $E_H$ admits a logarithmic connection singular along $D$, then $E_H(G)$ admits a logarithmic connection 
	singular along $D$. The converse holds if $H$ is reductive. 
\end{proposition}
The case of parabolic subgroup $P$ of a reductive affine algebraic group $G$ over $\bb C$ is interesting. 
Let $L \cong P/R_u(P)$ be the Levi factor of $P$, where $R_u(P)$ is the unipotent radical of $P$. 
Let $E_L$ be the corresponding holomorphic principal $L$-bundle on $X$ obtained by extending the structure 
group from $P$ to $L$. The natural action of $P$ on the Lie algebra $\mf n := {\rm Lie}(R_u(P))$ 
give rise to a holomorphic vector bundle $E_P(\mf n) := E_P\times^P\mf n$ on $X$. 
Then we have the following (see Theorem \ref{thm:parabolic-case}). 
\begin{theorem} 
	Suppose that $H^1(X, E_P(\mf n)\otimes\Omega_X^1(\log D)) = 0$. Then $E_P$ admits a logarithmic connection 
	singular along $D$ if $E_L$ admits a logarithmic connection singular along $D$. 
\end{theorem}

In Section \S\ref{sec-restriction-thms}, we discuss how existence of logarithmic connection on $E_G$ 
singular along $D$ can be ensured from existence of logarithmic connection on $E_G\big\vert_{X_n}$, 
where $X_n$ is some sufficiently high degree hypersurface in $X$ intersecting $D$ properly. 
More precisely, we fix an embedding $X \hookrightarrow \bb{CP}^N$, for some $N > 0$. 
By a hypersurface $X_n$ of degree $n$ in $X$, we mean $X \cap H_n$, for some 
hypersurface $H_n$ in $\bb{CP}^N$ of degree $n$. 
In \cite[Proposition 21]{Atiyah-1957}, it is shown that if $\dim_{\bb C}(X) \geq 3$, then $E_G$ admits a 
holomorphic connection if and only if for some smooth hypersurface $X_n$ in $X$ of sufficiently large degree, 
the principal $G$-bundle $E_G\big\vert_{X_n}$ admits a holomorphic connection. 
However, it is shown in \cite{Atiyah-1957} that this result fails if $\dim_{\bb C}(X) = 2$; 
see also \cite{Biswas-Gurjar-2018}. 
Also there are no complete answers known for this problem if $\dim_{\bb C}(X) = 2$. 
We prove the following analogue of \cite[Proposition 21]{Atiyah-1957} in the case of logarithmic 
connections on $E_G$ singular along $D$ in $X$ (see Theorem \ref{thm-1}). 

\begin{theorem}
	With the above notations, if $\dim_{\bb C}(X) \geq 3$ and $D \subset X$ a reduced effective divisor in $X$, 
	then $E_G$ admits a logarithmic connection singular along $D$ if and only if for some smooth 
	hypersurface $X_n$ in $X$ of sufficiently large degree $n$, intersecting $D$ properly, the principal 
	$G$-bundle $E_G\big\vert_{X_n}$ on $X_n$ admits a logarithmic connection singular along $D \cap X_n$. 
\end{theorem}

\section{Preliminaries}\label{sec:preli}
\subsection{Simple normal crossing divisor}\label{sec:SNC}
Let $X$ be a connected smooth complex projective variety of dimension at least one. 
We denote by $TX$ (respectively, $\Omega_X^1$) the tangent bundle (respectively, cotangent bundle) of $X$. 
The ideal sheaf $\mathscr{I}_D$ of an effective divisor $D$ on $X$ is a line bundle on $X$, 
denoted $\mc O_X(-D)$. 
\begin{definition}
	An effective divisor $D$ on $X$ is said to be a {\it simple normal crossing divisor} if $D$ is reduced, 
	each irreducible components of $D$ are smooth, and for each point $x \in X$, there is a system of regular 
	elements (local parameters) $z_1, \ldots, z_n \in \mf m_x$ such that the stalk $\mc O_X(-D)_x$ of the line bundle 
	$\mc O_X(-D)$ at $x$ is generated by the product $z_1\cdots z_r$, for some integer $r$ with $1 \leq r \leq n$. 
\end{definition}
In other words, a {\it simple normal crossing divisor} on $X$ is a reduced effective divisor $D$ on $X$, 
all of whose irreducible components are smooth, and locally for some choice of coordinate functions $(z_1,\ldots, z_n)$ 
around a point $x_0 \in U \subset X$, $D \cap U$ is given by an equation $z_1 \cdots z_r = 0$, for some integer $r$ 
with $1 \leq r \leq n$. This means, the irreducible components of $D$ passing through $x_0$ are given by the equations 
$z_i = 0$, for $i = 1, \ldots, r$, and they intersects each others transversally. 

\subsection{Logarithmic connection}
Let $D \subset X$ be a reduced effective divisor on $X$. 
For an integer $p \geq 0$, a {\it meromorphic $p$-form} on $X$ is a section of 
$\Omega_X^p(D) := \Omega_X^p \otimes_{\mc O_X} \mc O_X(D)$. 
A meromorphic $p$-form $\alpha \in (\Omega_X^p(D))(U)$ on an open subset $U \subset X$ is said to have 
a {\it logarithmic pole along $D$} if $\alpha$ is holomorphic on $U \setminus (U\cap D)$ and $\alpha$ 
has pole of order at most one along each irreducible component of $D$, and the same holds for $d\alpha$, 
where $d$ denotes the holomorphic exterior differential operator (see \cite[p.~197]{Voisin-I}). 
Let $\Omega_X^p(\log D)$ be the subsheaf of meromorphic $p$-forms on $X$ with at most logarithmic pole along $D$. 

Let $p : E \to X$ be a holomorphic vector bundle of rank $r$ on $X$. 
By abuse of notation, we denote by $E$ the sheaf of holomorphic sections of $p : E \to X$; 
this is a locally free coherent sheaf of $\mc O_X$-modules of rank $r$ on $X$. 
\begin{definition}\label{log-conn-def}
	A {\it logarithmic connection} on $E$ singular along $D$ is a $\bb C$-linear sheaf homomorphism 
	$$\nabla : E \longrightarrow E \otimes_{\mc O_X} \Omega_X^1(\log D)$$ 
	satisfying the Leibniz rule 
	$$\nabla(f\cdot s) = f\nabla(s) + s \otimes df,$$ 
	for all locally defined section $f$ of $\mc O_X$ and locally defined section $s$ of $E$. 
\end{definition}

\subsection{Residue of a logarithmic connection}\label{sec:residue-log-conn-vb}
We now recall the definition of reside of a logarithmic connection from \cite{Deligne-1970, Ohtsuki-1982}. 
Let $D$ be a simple normal crossing divisor on $X$. 
Write $D = \bigcup\limits_{j\in J} D_j$ as a union of all of its irreducible components. 
Let $E$ be a holomorphic vector bundle of rank $r$ on $X$ admitting a logarithmic connection 
$$\nabla : E \longrightarrow E \otimes_{\mc O_X} \Omega_X^1(\log D)$$ 
singular along $D$. 
Since each irreducible component $D_j$ of $D$ are smooth, using {\it Poincar\'e residue map} 
(see \cite[p.~211]{Voisin-I}, \cite[p.~147]{Griffiths-Harris-1994}), we have the following homomorphism 
\begin{equation*}
{\rm Res}_{D_j} : E \otimes \Omega_X^1(\log D) \longrightarrow E\otimes \mc O_{D_j}\,, \ \ \forall\ j. 
\end{equation*}
Then the composite map 
\begin{equation}\label{residue-map}
	{\rm Res}_{D_j}\circ\nabla : E\big\vert_{D_j} \longrightarrow E\big\vert_{D_j}
\end{equation} 
is a $\mc O_{D_j}$-module homomorphism, and hence defines a section 
$${\rm Res}_{D_j}(\nabla) \in H^0(D_j, \sEnd(E)\big\vert_{D_j})\,,$$
called the {\it residue of $\nabla$ along $D_j$}. 
For the sake of completeness, we recall explicit description of the reside of $\nabla$ along $D_j$ 
using local coordinates; \cite{Ohtsuki-1982}. 

Since $D$ is a simple normal crossing divisor on $X$, we can choose an open cover 
$\{U_\lambda : \lambda \in \Lambda\}$ of $X$ such that for each $\lambda \in \Lambda$, 
\begin{enumerate}[(I)]\label{local-exp-residue}
	\item $E\big\vert_{U_\lambda}$ is trivial, and 
	\item for each irreducible component $D_j$ of $D$, with $D_j \cap U_\lambda \neq \emptyset$, 
	we can choose a local coordinate function $f_{\lambda j} \in \mc O_X(U_\lambda)$ for a local 
	coordinate system on $U_\lambda$, such that $f_{\lambda j}$ is a defining equation of $D_j \cap U_\lambda$. 
	If $D_j \cap U_\lambda = \emptyset$, we take $f_{\lambda j} = 1$. 
\end{enumerate}

If $\nabla_\lambda$ is the connection matrix of $\nabla$ with respect to a holomorphic local frame 
$s_\lambda = (s_{\lambda 1}, \ldots, s_{\lambda r})$ for $E$ on $U_\lambda$, then we have 
\begin{equation}
	\nabla(s_\lambda) = \nabla_\lambda \otimes s_\lambda, 
\end{equation}
where $\nabla_\lambda$ is a $r\times r$ matrix whose entries are holomorphic sections of 
$\Omega_X^1(\log D)$ over $U_\lambda$. For each $D_j$, the matrix $\nabla_\lambda$ can be written as 
\begin{equation}
	\nabla_\lambda = R_{\lambda j} \frac{df_{\lambda j}}{f_{\lambda j}} + S_{\lambda j}\,, 
\end{equation}
where $R_{\lambda j}$ is a $r\times r$ matrix with entries in $\mc O_X(U_{\lambda})$ and $S_{\lambda j}$ is a 
$r\times r$ matrix with entries in $(\Omega_X^1(\log D))(U_\lambda)$ with simple pole along 
$\bigcup\limits_{j' \neq j} D_{j'}$. Then 
\begin{equation}
	{\rm Res}_{D_j}(\nabla_\lambda) := R_{\lambda j}\big\vert_{U_\lambda\cap D_j}\, 
\end{equation}
is a $r\times r$ matrix whose entries are holomorphic functions on $U_\lambda\cap D_j$; it is independent of 
choice of local defining equation $f_{\lambda j}$ for $D_j$. 
Then $\{{\rm Res}_{D_j}(\nabla_\lambda)\}_{\lambda\in \Lambda}$ defines a holomorphic global section 
\begin{equation}\label{residue}
	{\rm Res}_{D_j}(\nabla) \in H^0(D_j, \sEnd(E\big\vert_{D_j}))\,, 
\end{equation}
known as the {\it residue of $\nabla$ along $D_j$}. 

\begin{remark}
	If we further assume that intersections of any finite number of irreducible components of $D$ are connected, 
	then the Chern classes of $E$ can be computed in terms of the residues of the logarithmic connection $\nabla$ 
	along the irreducible components of $D$, and the first Chern classes of the line bundles associated to the 
	irreducible components of $D$; see \cite[Theorem 3, p.~16]{Ohtsuki-1982}. 
\end{remark}

\section{Logarithmic Connection On Principal Bundles}\label{sec:log-con-princ-bun}
\subsection{Logarithmic Atiyah exact sequence}\label{sec:log-At-ex-seq}
Let $G$ be a connected complex Lie group with Lie algebra $\mf g$. Let 
\begin{equation}\label{principal-bundle}
	p : E_G \longrightarrow X 
\end{equation}
be a holomorphic principal $G$-bundle on $X$. The holomorphic $G$-action on $E_G$ induces a 
holomorphic $G$-action on the holomorphic tangent bundle $TE_G$ of $E_G$, and the associated quotient 
$\At(E_G) := TE_G/G$ is a holomorphic vector bundle on $X$, known as the {\it Atiyah bundle} 
of $E_G$; the sections of $\At(E_G)$ are given by $G$-invariant holomorphic vector fields on $E_G$. 
Let $\ad(E_G) := E_G \times^G \mf g$ be the {\it adjoint vector bundle} associated to the adjoint 
representation of $G$ to its Lie algebra $\mf g$. 
The surjective submersion $p$ in \eqref{principal-bundle} induces a short exact sequence of holomorphic 
vector bundles on $X$, 
\begin{equation}\label{Atiyah-ext-seq}
	0 \longrightarrow \ad(E_G) \longrightarrow \At(E_G) \stackrel{d'p}{\longrightarrow} TX \longrightarrow 0, 
\end{equation}
called the {\it Atiyah exact sequence} of $E_G$. A {\it holomorphic connection} on $E_G$ is given by a holomorphic 
vector bundle homomorphism $\eta : TX \to \At(E_G)$ such that $d'p\circ \eta = \Id_{TX}$; see \cite{Atiyah-1957}. 
We now modify the exact sequence \eqref{Atiyah-ext-seq} to define a logarithmic Atiyah exact sequence. 

Let $D$ be a reduced effective divisor on $X$. 
Then $TX(-\log D) := (\Omega_X^1(\log D))^\vee$ is a locally free $\mc O_X$-submodule of $TX$. 
In fact, we have $TX(-D) \subseteq TX(-\log D) \subset TX$. 
Then we have a locally free $\mc O_X$-submodule $\mc A_D(E_G) := (d'p)^{-1}(TX(-\log D))$ of $\At(E_G)$ which fits 
into the following short exact sequence of locally free $\mc O_X$-modules  
\begin{equation}\label{log-Atiyah-ext-seq}
	0 \longrightarrow \ad(E_G) \stackrel{\iota_D}{\longrightarrow} \mc A_D(E_G) 
	\stackrel{\widetilde{d'p}}{\longrightarrow} TX(-\log D) \longrightarrow 0, 
\end{equation}
called the {\it logarithmic Atiyah exact sequence} of $E_G$ for the divisor $D$, (see also \cite{Biswas-Dan-Paul-2018}). 
Moreover, we have the following commutative diagram of $\mc O_X$-module homomorphisms 
\begin{equation}\label{log-hol-At-ext-seq}
\begin{gathered}
	\xymatrix{
		0 \ar[r] & \ad(E_G) \ar@{=}[d] \ar[r]^{\iota_D} & \mc A_D(E_G) \ar@{^(->}[d]^J 
		\ar[r]^-{\widetilde{d'p}} & TX(-\log D) \ar@{^(->}[d]^I \ar[r] & 0 \\ 
		0 \ar[r] & \ad(E_G) \ar[r]^{\iota} & \At(E_G) \ar[r]^-{d'p} & TX \ar[r] & 0 
	}
\end{gathered}
\end{equation} 

Let $E$ be a holomorphic vector bundle $E$ of rank $n$ on $X$. 
Let $p : E_{\GL_n(\bb C)} \longrightarrow X$ be the holomorphic frame bundle of $E$; 
this is a principal $\GL_n(\bb C)$-bundle on $X$. 
Note that, $\ad(E_{\GL_n(\bb C)})$ is naturally isomorphic to $\sEnd(E)$. 

\begin{proposition}\label{prop-1}
	$E$ admits a logarithmic connection $\nabla : E \to E \otimes \Omega_X^1(\log D)$ singular along $D$ if and only if 
	the exact sequence in \eqref{log-Atiyah-ext-seq} associated to $E_{\GL_n(\bb C)}$ splits holomorphically. 
\end{proposition}

\begin{proof}
	Let $G = \GL_n(\bb C)$. Let $\Der_{\bb C}(E_G)$ be the sheaf of $\bb C$-linear derivations of $\mc O_{E_G}$. 
	Then there is a natural $\mc O_{E_G}$-module isomorphism 
	$\Der_{\bb C}(E_G) \stackrel{\simeq}{\to} \sHom(\Omega_{E_G}^1, \mc O_{E_G}) = TE_G$ defined by sending 
	a locally defined $\bb C$-linear derivation $\xi$ of $\mc O_{E_G}$ to the unique $\mc O_{E_G}$-module 
	homomorphism $\widetilde{\xi} : \Omega^1_{E_G} \to \mc O_{E_G}$ such that $\widetilde{\xi}\circ d = \xi$, 
	where $d : \mc O_{E_G} \to \Omega_{E_G}^1$ is the K\"ahler differential operator on $E_G$. 
	Then the $G$-invariant sections of $\Der_{\bb C}(E_G)$ descend to sections of $\At(E_G)$. 
	
	Now it is clear that given a $\mc O_X$-module homomorphism 
	$\eta : TX(-\log D) \to \mc A_D(E)$ 
	with $\eta\circ\widetilde{d'p} = \Id_{TX(-\log D)}$, for each locally defined section $\xi$ of $TX(-\log D)$, 
	its image $\eta(\xi)$ defines a $G$-invariant $\bb C$-linear derivation of $E$. 
	Thus we have a logarithmic connection on $E$ singular along $D$. 
	Conversely, given a logarithmic connection $\nabla : E \to E\otimes\Omega_X^1(\log D)$ singular along $D$, for 
	each locally defined section $\xi$ of $TX(-\log D) = \sHom(\Omega_X^1(\log D), \mc O_X)$, we get a 
	$\GL_n(\bb C)$-invariant $\bb C$-linear derivation $\nabla_{\xi} := (\Id_E\otimes\xi)\circ\nabla$ of $E$. 
	This defines a splitting of the short exact sequence \eqref{log-Atiyah-ext-seq}. 
\end{proof}

The above Proposition \ref{prop-1} motivates us to define the following (see also \cite[\S 2.2]{Biswas-Dan-Paul-Saha-2017}).  
\begin{definition}
	Let $p : E_G \to X$ be a holomorphic principal $G$-bundle on $X$. 
	A {\it logarithmic connection} on $E_G$ singular along $D$ is a holomorphic vector bundle homomorphism 
	$\eta : TX(-\log D) \to \mc A_D(E_G)$ such that $\widetilde{d'p}\circ\eta = \Id_{TX(-\log D)}$, where 
	$\widetilde{d'p}$ is the homomorphism in \eqref{log-Atiyah-ext-seq}. 
\end{definition}
We refer the exact sequence \eqref{log-Atiyah-ext-seq} as the {\it logarithmic Atiyah exact sequence} of $E_G$ 
associated to the divisor $D$. 
The exact sequence \eqref{log-Atiyah-ext-seq} defines a cohomology class 
\begin{equation}\label{log-Atiyah-class}
\Phi_D(E) \in H^1(X, \ad(E_G)\otimes \Omega_X^1(\log D))\,, 
\end{equation}
which we call the {\it logarithmic Atiyah class} of $E$ along $D$, such that the exact sequence 
\eqref{log-Atiyah-ext-seq} splits holomorphically if and only if $\Phi_D(E) = 0$. 

\subsection{Residue of logarithmic connection on a principal bundle}
Let $D$ be a simple normal crossing divisor on $X$, locally defined by $z_1\cdots z_r = 0$. 
Let us denote by $D_j$ the irreducible component of $D$ locally defined by $z_j = 0$, for each $j = 1, \ldots, r$. 
Let $TX(-\log D)$ be the dual of $\Omega_X^1(\log D)$; this is a locally free coherent sheaf of $\mc O_X$-modules of 
rank $d = \dim_{\bb C}(X)$, with local frame fields given by 
$\big(z_1\frac{\partial}{\partial z_1}, \ldots, z_r \frac{\partial}{\partial z_r}, 
\frac{\partial}{\partial z_{r+1}}, \ldots, \frac{\partial}{\partial z_d}\big)$. 
For each $j = 1, \ldots, r$, over $D_j$, we can identify $z_j\frac{\partial}{\partial z_j}$ with $1$; 
this identification is independent of choice of local coordinate system $(z_1, \ldots, z_d)$ on $X$ 
such that $D_i$ is locally given by vanishing locus of $z_i$, for all $i = 1, \ldots, r$. 
Thus, $TX(-\log D)\big\vert_{D_j}$ is locally free $\mc O_{D_j}$-module generated by 
$$\left(z_1\frac{\partial}{\partial z_1}, \ldots, z_{j-1}\frac{\partial}{\partial z_{j-1}}, 1, 
z_{j+1}\frac{\partial}{\partial z_{j+1}}, \ldots, z_r \frac{\partial}{\partial z_r}, 
\frac{\partial}{\partial z_{r+1}}, \ldots, \frac{\partial}{\partial z_d}\right).$$ 
Therefore, we have an injective homomorphism 
$\mc O_{D_j} \longrightarrow TX(-\log D)\big\vert_{D_j}$. 
Let 
\begin{equation}\label{eta}
\eta : TX(-\log D) \longrightarrow \mc A_D(E_G) 
\end{equation}
be a logarithmic connection on $E_G$ singular along $D$; that means, $\eta$ is an $\mc O_X$-module 
homomorphism such that $\widetilde{d'p}\circ\eta = \Id_{TX(-\log D)}$ (see \eqref{log-Atiyah-ext-seq}). 
Note that the image of $\eta\big\vert_{\mc O_{D_j}}$ lands inside 
$\ad(E_G)\big\vert_{D_j} \subset \mc A_{D_j}(E_G)\big\vert_{D_j}$. 
This gives a section 
\begin{equation}\label{residue-principal}
{\rm Res}_{D_j}(\eta) \in H^0(D_j, \ad(E_G)\big\vert_{D_j}), 
\end{equation}
called the {\it residue} of $\eta$ along $D_j$, for all $j = 1, \ldots, r$. 
Then we have the following. 

\begin{proposition}
	Let $E$ be a holomorphic vector bundle of rank $n$ on $X$, and let $E_{\GL_n(\bb C)}$ be the holomorphic frame bundle of 
	$E$. If $\eta$ in \eqref{eta} is the logarithmic connection on $E_{\GL_n(\bb C)}$ associated to a logarithmic connection 
	$\nabla$ on $E$ as defined in \eqref{log-conn-def}, then for each irreducible component $D_j$ of $D$, we have 
	\begin{equation}
	{\rm Res}_{D_j}(\nabla) = {\rm Res}_{D_j}(\eta)\,,
	\end{equation}
	where ${\rm Res}_{D_j}(\nabla)$ is as defined in \eqref{residue} and 
	${\rm Res}_{D_j}(\eta)$ is as defined in \eqref{residue-principal}. 
\end{proposition}

\begin{proof}
	Follows from the proof of Proposition \ref{prop-1} and the definition of residue in \eqref{residue-map}. 
\end{proof}

\subsection{Extension of structure group}\label{sec:extn-str-gp}
Let $G$ and $H$ be two connected complex Lie groups with Lie algebras $\mf g$ and $\mf h$, respectively. 
Let $f : H \longrightarrow G$ be a homomorphism of complex Lie groups, and $df : \mf h \longrightarrow \mf g$ 
the Lie algebra homomorphism induced by $f$. Let $p : E_H \to X$ be a holomorphic principal $H$-bundle on $X$. 
Then we have a holomorphic principal $G$-bundle $p' : E_G := E_H(G) \to X$ on $X$ obtained by extending the 
structure group of $E_H$ by the homomorphism $f$. Then there is a natural vector bundle homomorphisms 
$\alpha : \ad(E_H) \longrightarrow \ad(E_G)$ and $\beta : \At(E_H) \longrightarrow \At(E_G)$ 
induced by $f$. Then we have the following commutative diagram of vector bundle homomorphisms 
with two rows exact (see \cite{Atiyah-1957}). 
\begin{equation}\label{Hol-At-H-G}
\begin{gathered}
\xymatrix{
	0 \ar[r] & \ad(E_H) \ar[d]^{\alpha} \ar[r] & \At(E_H) \ar[d]^\beta \ar[r]^-{d'p} & TX \ar@{=}[d] \ar[r] & 0 \\ 
	0 \ar[r] & \ad(E_G) \ar[r] & \At(E_G) \ar[r]^-{d'p'} & TX \ar[r] & 0} 
\end{gathered}
\end{equation}
Let $D$ be a reduced effective divisor on $X$. 
Then the commutative diagram \eqref{log-hol-At-ext-seq} and \eqref{Hol-At-H-G} gives the following 
commutative diagram of vector bundle homomorphisms with two rows exact. 
\begin{equation}\label{log-At-H-G}
\begin{gathered}
\xymatrix{
	0 \ar[r] & \ad(E_H) \ar[d]^{\alpha} \ar[r] & \mc A_D(E_H) \ar[d]^{\beta} \ar[r]^-{\widetilde{d'p}} & 
	TX(-\log D) \ar@{=}[d] \ar[r] & 0 \\ 
	0 \ar[r] & \ad(E_G) \ar[r] & \mc A_D(E_G) \ar[r]^-{\widetilde{d'p'}} & TX(-\log D) \ar[r] & 0 } 
\end{gathered}
\end{equation}
If $\eta : TX(-\log D) \to \mc A_D(E_H)$ is a holomorphic vector bundle homomorphism with 
$\widetilde{d'p}\circ\eta = \Id_{TX(-\log D)}$, then $f_*(\eta) := \beta\circ\eta$ satisfies 
$\widetilde{d'p'}\circ(f_*\eta) = \Id_{TX(-\log D)}$. 
Consequently, if $D$ is a simple normal crossing divisor $D$ in $X$, for each irreducible component $D_j$ of $D$, 
we have ${\rm Res}_{D_j}(f_*\eta) = \alpha\circ{\rm Res}_{D_j}(\eta)$; 
(see also \cite[\S 2.4]{Biswas-Dan-Paul-Saha-2017}). 

In fact, it follows from commutativity of the diagram \eqref{log-At-H-G} that there is a natural homomorphism of 
cohomologies 
\begin{equation}\label{log-At-cls-H-G}
f_* : H^1(X, \ad(E_H) \otimes \Omega_X^1(\log D)) \longrightarrow H^1(X, \ad(E_G) \otimes \Omega_X^1(\log D)), 
\end{equation}
induced by $f$, which sends the cohomology class $\Phi_D(E_H)$ to $\Phi_D(E_G)$; see \eqref{log-Atiyah-class}. 
Since the homomorphism \eqref{log-At-cls-H-G} is not necessarily injective, in general, existence of a logarithmic 
connection on $E_G$ singular along $D$ may not ensure existence of a logarithmic connection on $E_H$ singular along $D$. 
However, if $f : H \longrightarrow G$ is an injective homomorphism of connected affine algebraic groups over $\bb C$ 
with $H$ reductive, then the above homomorphism \eqref{log-At-cls-H-G} can be shown to be injective 
(see the proof of \cite[Lemma 3.3]{Biswas-Dan-Paul-Saha-2017} for more details). 
Therefore, from the above discussions, we have the following. 
\begin{proposition}\label{prop:red-str-gp}
	With the above notations, $E_G$ admits a logarithmic connection singular along $D$ if $E_H$ admits a 
	logarithmic connection singular along $D$. Converse holds if $f : H \to G$ is an injective homomorphism 
	of connected affine algebraic groups over $\bb C$ with $H$ reductive. 
\end{proposition}

Let $G$ be a connected reductive affine algebraic group over $\bb C$. 
Let $P$ be a parabolic subgroup of $G$. Let $R_u(P)$ be the unipotent radical of $P$. 
Then there is a closed connected algebraic subgroup $L \subset P$ such that the 
restriction to $L$ of the quotient homomorphism 
\begin{equation}
	q : P \longrightarrow P/R_u(P)\,,  \nonumber
\end{equation}
is an isomorphism of algebraic groups over $\bb C$. Clearly, $L$ is reductive; and it is known as the 
{\it Levi factor} of $P$ (see e.g., \cite[p.~559]{Milne-2017}). 
Consider the homomorphism 
\begin{equation}\label{eqn:Levi-map}
	q' := \big(q\big\vert_L\big)^{-1}\circ q : P \longrightarrow L\,. 
\end{equation}
Let $E_P$ be a homomorphic principal $P$-bundle on $X$. Let $E_L := E_P(L)$ be the holomorphic principal 
$L$-bundle on $X$ obtained by extending the structure group of $E_P$ by the homomorphism $q'$ in 
\eqref{eqn:Levi-map}. The Lie algebra $\mf n := {\rm Lie}(R_u(P))$ of $R_u(P)$ is the nilpotent radical 
of the Lie algebra $\mf p := {\rm Lie}(P)$ of $P$. The action of $P$ on $\mf n$ gives rise to a holomorphic 
vector bundle $E_P(\mf n) := E_P\times^P\mf n$ on $X$. 
Note that, $E_P(\mf n)$ is a subbundle of $\ad(E_P) = E_P(\mf p)$, and the associated quotient 
vector bundle $\ad(E_P)/E_P(\mf n)$ is isomorphic to $E_P(\mf l) \cong \ad(E_L)$, where 
$\mf l = {\rm Lie}(L)$. Then we have the following. 
\begin{theorem}\label{thm:parabolic-case} 
	With the above notations, if $H^1(X, E_P(\mf n)\otimes\Omega_X^1(\log D)) = 0$, then 
	$E_P$ admits a logarithmic connection singular along $D$ whenever $E_L$ admits a logarithmic connection 
	singular along $D$. 
\end{theorem}

\begin{proof}
	Replacing $H$ by $P$ and $G$ by $L$ in the commutative diagram \eqref{log-At-H-G}, we have the following 
	commutative diagram of holomorphic vector bundle homomorphisms, with all rows and columns exact. 
	\begin{equation}\label{log-At-P-L}
	\begin{gathered}
	\xymatrix{
		\, & 0 \ar[d] & 0 \ar[d] & \, & \\ 
		\, & E_P(\mf n) \ar[d] \ar@{=}[r] & E_P(\mf n) \ar[d] & \, & \, \\ 
		0 \ar[r] & \ad(E_P) \ar[d]^{\alpha} \ar[r] & \mc A_D(E_P) \ar[d]^{\beta} \ar[r]^-{\sigma_P} & 
		TX(-\log D) \ar@{=}[d] \ar[r] & 0 \\ 
		0 \ar[r] & \ad(E_L) \ar[d] \ar[r] & \mc A_D(E_L) \ar[d] \ar[r]^-{\sigma_L} & TX(-\log D) \ar[r] & 0 \\ 
		\, & 0 & 0 & \, & \, 
		} 
	\end{gathered}
	\end{equation}
	Let $\eta : TX(-\log D) \to \mc A_D(E_L)$ be an $\mc O_X$-module homomorphism such that 
	$\sigma_L\circ\eta = \Id_{TX(-\log D)}$, where $\sigma_L$ is the homomorphism in \eqref{log-At-P-L}. 
	Let $\mc F := \beta^{-1}(\eta(TX(-\log D))) \subset \mc A_D(E_P)$. This fits into the following 
	short exact sequence of $\mc O_X$-modules 
	\begin{equation}\label{eqn2}
		0 \longrightarrow E_P(\mf n) \longrightarrow \mc F \longrightarrow TX(-\log D) \longrightarrow 0\,. 
	\end{equation}
	Then the logarithmic Atiyah exact sequence for $E_P$ in \eqref{log-At-P-L} splits $\mc O_X$-linearly 
	if the exact sequence \eqref{eqn2} splits $\mc O_X$-linearly. Since the obstruction for splitting of 
	the exact sequence \eqref{eqn2} lies in $H^1(X, E_P(\mf n)\otimes\Omega_X^1(\log D))$, 
	the result follows. 
\end{proof}

\section{Existence of Logarithmic Connection}\label{sec-restriction-thms}
\subsection{Restriction theorem for logarithmic connection}
Let $X$ be a smooth complex projective variety of dimension $d \geq 1$. Fix an embedding of $X$ into a complex 
projective space $\bb{CP}^N$, for some positive integer $N$. A hypersurface $X_n$ of degree $n$ in $X$ is given 
by $X \cap H_n$, where $H_n$ is a hypersurface of degree $n$ in $\bb{CP}^N$. 
For general hypersurfaces $H_n$, we get $X_n = X \cap H_n$ smooth \cite{Hartshorne-1977}. 
Let ${\rm Div}(X)$ be the group of all divisors in $X$. For $D_1, D_2 \in {\rm Div}(X)$, we say that $D_1$ and $D_2$ 
{\it meets properly} if for each prime divisor $V$ (respectively, $W$) appearing with non-zero coefficient in $D_1$ 
(respectively, $D_2$), we have $\dim (V \cap W) = d-2$. 
It is clear that if two reduced effective divisors $D_1, D_2 \in {\rm Div}(X)$ meets properly, 
then $D_1 \cap D_2$ is a divisor in both $D_1$ and $D_2$. 

Let $G$ be a connected complex Lie group, and $E_G$ a holomorphic principal $G$-bundle on $X$. 
Then we have the following result. 

\begin{theorem}\label{thm-1}
	Assume that $\dim_{\bb C}(X) \geq 3$ and $D \subset X$ a reduced effective divisor in $X$. 
	Then $E_G$ admits a logarithmic connection singular along $D$ if and only if for some smooth 
	hypersurface $X_n$ of sufficiently large degree $n$, which intersects $D$ properly, the principal 
	$G$-bundle $E_G\big\vert_{X_n}$ on $X_n$ admits a logarithmic connection singular along $D \cap X_n$. 
\end{theorem}

\begin{proof}
	For any divisor $H$ on $X$, we denote by $\mc O_X(H)$ the line bundle on $X$ associated to $H$. 
	Let $\mc F$ be a coherent sheaf of $\mc O_X$-modules on $X$. 
	Consider the exact sequence of sheaves 
	\begin{equation}\label{ext-seq-1}
		0 \longrightarrow \mc F \otimes\mc O_X(-X_n) \longrightarrow \mc F \longrightarrow 
		{\iota_n}_*(\mc F\big\vert_{X_n}) \longrightarrow 0\,, 
	\end{equation}
	where $\iota_n : X_n \hookrightarrow X$ is the inclusion morphism. 
	Since $\dim_{\bb C}(X) \geq 3$, it follows from Serre's theorem \cite[p.~228]{Hartshorne-1977} that for $n \gg 0$, 
	we have 
	\begin{equation}\label{eqn-4}
		H^i(X, \mc F \otimes \mc O_X(-X_n)) = 0\,, \ \ \ \forall \ i = 1, 2. 
	\end{equation}
	Then the long exact sequence of cohomologies associated to the short exact sequence \eqref{ext-seq-1} 
	gives an isomorphism. 
	\begin{equation}\label{eqn-3}
		H^1(X, \mc F) \stackrel{\cong}{\longrightarrow} H^1(X_n, \mc F\big\vert_{X_n})\,. 
	\end{equation}
	Since $X_n$ intersects $D$ properly by assumption, $D_n := X_n \cap D$ is an effective divisor 
	in $X_n$, and we have a natural isomorphism $\mc O_X(D)\big\vert_{X_n} \cong \mc O_{X_n}(D_n)$. 
	Then from \cite[Chapter II, Theorem 8.17]{Hartshorne-1977}, we have an exact sequence of $\mc O_{X_n}$-modules 
	\begin{equation}\label{eqn-1}
		0 \longrightarrow \left(\scr{I}_{X_n}/\scr{I}_{X_n}^2\right) \otimes \mc O_{X_n}(D_n) \longrightarrow 
		\Omega_X^1(D)\big\vert_{X_n} \stackrel{\xi}{\longrightarrow} \Omega_{X_n}^1(D_n) \longrightarrow 0\,,
	\end{equation}
	where $\scr{I}_{X_n}$ is the ideal sheaf of the hypersurface $X_n$ in $X$. 
	Note that there is a natural $\mc O_{X_n}$-module isomorphism 
	\begin{equation}\label{eqn-2}
		\Omega_X^1(\log D)\big\vert_{X_n} \stackrel{\cong}{\longrightarrow} \xi^{-1}(\Omega_{X_n}^1(\log D_n))\,.  
	\end{equation}
	Since $\scr{I}_{X_n}/\scr{I}_{X_n}^2 \cong \mc O_X(-X_n)\big\vert_{X_n}$, from \eqref{eqn-1} using \eqref{eqn-2} 
	we have the following short exact sequence of $\mc O_{X_n}$-modules 
	\begin{equation}\label{eqn-6}
		0 \longrightarrow \mc O_X(D-X_n)\big\vert_{X_n} \longrightarrow 
		\Omega_X^1(\log D)\big\vert_{X_n} \longrightarrow \Omega_{X_n}^1(\log D_n) \longrightarrow 0\,.
	\end{equation}
	Now tensoring the exact sequence \eqref{eqn-6} with $\ad(E_G)\big\vert_{X_n}$, we get the following 
	short exact sequence of $\mc O_{X_n}$-modules 
	\begin{eqnarray}\label{eqn-7}
	0 \longrightarrow \left(\ad(E_G)\otimes\mc O_X(D-X_n)\right)\big\vert_{X_n} \longrightarrow 
	\left(\ad(E_G)\otimes\Omega_X^1(\log D)\right)\big\vert_{X_n} \nonumber \\ 
	\longrightarrow \ad(E_G)\big\vert_{X_n}\otimes\Omega_{X_n}^1(\log D_n) \longrightarrow 0\,.
	\end{eqnarray}
	Now taking $\mc F = \ad(E_G) \otimes\mc O_X(D)$ and $\mc F = \ad(E_G) \otimes\mc O_X(D-X_n)$ 
	in \eqref{eqn-4}, we get 
	\begin{equation}\label{eqn-5}
		H^1(X, \ad(E_G) \otimes\mc O_X(D-X_n)) = 0 = H^2(X, \ad(E_G) \otimes\mc O_X(D-2X_n))\,, 
	\end{equation}
	for $n$ large enough. Fix one such $n \gg 0$. 
	Then applying \eqref{eqn-3} for $\mc F = \ad(E_G) \otimes \mc O_X(D-X_n)$, 
	using \eqref{eqn-5} we get 
	\begin{equation}\label{eqn-8}
		H^1(X_n, \left(\ad(E_G)\otimes\mc O_X(D-X_n)\right)\big\vert_{X_n}) = 0\,. 
	\end{equation}
	Now from the long exact sequence of cohomologies associated to \eqref{eqn-7}, using \eqref{eqn-8} we get 
	an exact sequence of cohomologies 
	\begin{equation}\label{eqn-9}
		0 \longrightarrow H^1(X_n, \left(\ad(E_G)\otimes\Omega_X^1(\log D)\right)\big\vert_{X_n}) 
		\longrightarrow H^1(X_n, \ad(E_G\big\vert_{X_n})\otimes\Omega_{X_n}^1(\log D_n))\,. 
	\end{equation}
	Now taking $\mc F = \ad(E_G)\otimes\Omega_X^1(\log D)$ in \eqref{eqn-3}, from \eqref{eqn-9} we see 
	that the inclusion map $\iota_n : X_n \hookrightarrow X$ induces an injective homomorphism 
	\begin{equation}\label{eqn-10}
		\widetilde{\iota_n} : H^1(X, \ad(E_G)\otimes\Omega_X^1(\log D)) 
		\longrightarrow H^1(X_n, \ad(E_G\big\vert_{X_n}) \otimes \Omega_{X_n}^1(\log D_n)) \,. 
	\end{equation}
	The inclusion morphism $\iota_n : X_n \hookrightarrow X$ induces the following commutative diagram of 
	homomorphisms of sheaves of $\mc O_X$-modules on $X$ with two rows exact. 
	\begin{equation*}
		\xymatrix{
		0 \ar[r] & \ad(E_G) \ar[d] \ar[r] & \mc A_D(E_G) \ar[d] \ar[r] & TX(-\log D) \ar[d] \ar[r] & 0 \\
		0 \ar[r] & {\iota_n}_*(\ad(E_G\big\vert_{X_n})) \ar[r] & {\iota_n}_*(\mc A_{D_n}(E_G\big\vert_{X_n})) 
		\ar[r] & {\iota_n}_*(TX_n(-\log D_n)) \ar[r] & 0 
			}
	\end{equation*}
	Now one can check that the homomorphism \eqref{eqn-10} sends the cohomology class 
	$\Phi_D(E_G) \in$ $H^1(X, \ad(E_G)\otimes\Omega_X^1(\log D))$, 
	as defined in \eqref{log-Atiyah-class}, 
	to the cohomology class $\Phi_{D_n}(E_G\big\vert_{X_n})$. 
	Thus $\Phi_D(E_G) = 0$ if and only if $\Phi_{D_n}(E_G\big\vert_{X_n}) = 0$. 
	This completes the proof. 
\end{proof}

\section*{Acknowledgements}
The authors would like to thank Indranil Biswas for suggesting this question. 
The first named author would like to thank the ``Department of Science and Technology'' of 
the Government of India for the INSPIRE fellowship - 15DSTINS009. 
The second named author would like to thank Arideep Saha for some useful discussions.


\begin{thebibliography}{AAAAAA}
	\providecommand{\doi}[2][]{doi: \href{https://doi.org/#2}{#2}}
	\providecommand{\arxiv}[2][]{arXiv:\href{https://arxiv.org/abs/#2}{#2}}
	
	\bibitem[AB02]{Azad-Biswas-2002}
	Hassan Azad and Indranil Biswas, On holomorphic principal bundles over a
	compact Riemann surface admitting a flat connection, \textit{Math. Ann.},
	\textbf{322}, no.~2 (2002), 333--346. \doi{10.1007/s002080100273}.
	
	\bibitem[Ati57]{Atiyah-1957}
	M. F. Atiyah, Complex analytic connections in fibre bundles, 
	\textit{Trans. Amer. Math. Soc.}, \textbf{85} (1957), 181--207. 
	\doi{10.2307/1992969}. 
	
	\bibitem[BDP18]{Biswas-Dan-Paul-2018}
	Indranil Biswas, Ananyo Dan and Arjun Paul, Criterion for logarithmic
	connections with prescribed residues, \textit{Manuscripta Math.},
	\textbf{155}, no. 1-2 (2018), 77--88. 
	\doi{10.1007/s00229-017-0935-6}.
	
	\bibitem[BDPS17]{Biswas-Dan-Paul-Saha-2017}
	Indranil Biswas, Ananyo Dan, Arjun Paul and Arideep Saha, Logarithmic
	connections on principal bundles over a Riemann surface, 
	\textit{Internat. J. Math.}, \textbf{28}, no.~12 (2017), 1750088, 18 pp. 
	\doi{10.1142/S0129167X17500884}. 
	
	\bibitem[BG18]{Biswas-Gurjar-2018} 
	Indranil Biswas and Sudarshan Gurjar, Connections and restrictions to curves, 
	\textit{C. R. Math. Acad. Sci. Paris} \textbf{356} (2018), no. 6, 674--678. 
	\doi{10.1016/j.crma.2018.05.004}. 
	
	\bibitem[Del70]{Deligne-1970}
	Pierre Deligne, \textit{\'Equations diff\'erentielles \`a points singuliers r\'eguliers}, 
	Lecture Notes in Mathematics, Vol. 163, Springer-Verlag, Berlin-New York (1970). 
	\doi{10.1007/BFb0061194}. 
	
	\bibitem[GH94]{Griffiths-Harris-1994}
	Phillip Griffiths and Joseph Harris, \textit{Principles of algebraic geometry},
	Wiley Classics Library, John Wiley \& Sons, Inc., New York (1994). 
	\doi{10.1002/9781118032527}. 
	
	\bibitem[Har77]{Hartshorne-1977}
	Robin Hartshorne, \textit{Algebraic geometry}, Graduate Texts in Mathematics, No. 52. 
	\textit{Springer-Verlag, New York-Heidelberg}, 1977.  
	\doi{10.1007/978-1-4757-3849-0}. 
	
	\bibitem[Mil17]{Milne-2017}
	J. S. Milne, \textit{Algebraic groups}, The theory of group schemes of finite type over a field, 
	Cambridge Studies in Advanced Mathematics, 170. \textit{Cambridge University Press, Cambridge}, 2017. 
	\doi{10.1017/9781316711736}. 
	
	\bibitem[Oht82]{Ohtsuki-1982}
	Makoto Ohtsuki, A residue formula for Chern classes associated with logarithmic connections, 
	\textit{Tokyo J. Math.}, \textbf{5}, no.~1 (1982), 13--21. 
	\doi{10.3836/tjm/1270215030}. 
	
	\bibitem[Voi07]{Voisin-I}
	Claire Voisin, \textit{Hodge theory and complex algebraic geometry. I}, 
	\textit{Cambridge Studies in Advanced Mathematics}, volume~76, 
	Cambridge University Press, Cambridge, english edition (2007). 
	\doi{10.1017/CBO9780511615344}. 
	
	\bibitem[Wei38]{Weil-1938}
	Andr\'e Weil, Généralisation des fonctions abéliennes, 
	\textit{J. Math. Pures Appl.}, \textbf{17} (1938), 47--87. 
	
\end{thebibliography}
\end{document}